%% file: main.tex
\documentclass{article}
\input{makros}

\usepackage
{graphicx, color}
\usepackage[utf8]{inputenc} 
\usepackage[T1]{fontenc}
\usepackage{url}
\usepackage{pdfpages} 
\usepackage[nottoc, notbib, notindex]{tocbibind}  
\usepackage{titletoc}
\usepackage{CJK}
\usepackage{graphicx}
\usepackage{tikz}
\usepackage{hyperref}
\usepackage{turnstile}
\definecolor{teal2}{rgb}{0.036, 0.512, 0.512}
\definecolor{veryblue}{rgb}{0.12, 0.33, 0.55}
\hypersetup{colorlinks=true}
\hypersetup{citecolor=veryblue,linkcolor=veryblue}
\usepackage{amsmath, enumitem,amscd, upgreek, stmaryrd, mathtools, mathrsfs}
\usepackage{amssymb}
\usepackage{amsthm}
\usepackage{amsfonts}
\usepackage{titling}
\theoremstyle{plain}
\newtheorem{thm}{Theorem}[section]
\newtheorem{mainthm}{Main Theorem}
\newtheorem{lem}[thm]{Lemma}

\newtheorem{cor}[thm]{Corollary}

\theoremstyle{definition}
\newtheorem{defn}{Definition}[section]

\theoremstyle{remark}

\newtheorem*{notation}{\bf Notation}

\newtheorem{fact}{\bf Fact}

\title{Martin's Maximum\texorpdfstring{${}^{*, ++}_{\mathfrak{c}}$}{Lg} in \texorpdfstring{$\mathbb{P}_{\max}$}{Lg} extensions of strong models of determinacy}
\author{Ralf Schindler\thanks{Institut für Mathematische Logik und Grundlagenforschung, Universität Münster, Einsteinstr. 62, 48149 Münster, FRG. Funded by the Deutsche Forschungsgemeinschaft (DFG, German Research Foundation) under Germany's Excellence Strategy EXC 2044 –390685587, Mathematics Münster: Dynamics–Geometry–Structure. Email: rds@uni-muenster.de} \and
 Taichi Yasuda\thanks{Institut für Mathematische Logik und Grundlagenforschung, Universität Münster, Einsteinstr. 62, 48149 Münster, FRG. Funded by the Deutsche Forschungsgemeinschaft (DFG, German Research Foundation) under Germany's Excellence Strategy EXC 2044 –390685587, Mathematics Münster: Dynamics–Geometry–Structure. Email: tyasuda@uni-muenster.de}}
\date{\today}

\begin{document}
\maketitle
\begin{abstract}
We study a strengthening of $\MM^{++}$ which is called $\MM^{*, ++}$ and which was introduced in \cite{DR3} and \cite{Sch}. We force its bounded version $\MM^{*, ++}_{\mathfrak{c}}$, which is stronger than both $\MM^{++}(\mathfrak{c})$ as well as $\BMM^{++}$, by $\mathbb{P}_{\max}$ forcing over a determinacy model $L^{F_{\rm uB}}({\mathbb R}^*,\mbox{Hom}^*)$. The construction of the ground model $L^{F_{\rm uB}}({\mathbb R}^*,\mbox{Hom}^*)$ builds upon \cite{TG} and the derived model construction of \cite{uB}. 
\end{abstract}
\input{Introduction}

{\large{\bf Notation.}} 
\begin{itemize}
    \item Let $\R=\omega^{\omega}$ as usual. 
    \item For $\alpha<\omega_{1}$, let $\mathrm{WO}_{\alpha}$ be the set of reals coding $\alpha$. Let $\mathrm{WO}=\bigcup_{\al<\omega_{1}}\mathrm{WO}_{\al}$. For more details, see \cite[4A]{Mos}. 
    \item Let $\On$ be the class of all the ordinals. 
    \item Let $\mathcal{L}_{\dot{\in}}$ be the language of set theory, and let $\mathcal{L}_{\dot{\in}, \dot{I}_{\NS}}$ be the language of set theory augmented by a predicate $\dot{I}_{\NS}$ for $\NS_{\omega_{1}}$. In transitive models $\mathcal{M}$ of $\ZFC^{-}+``\omega_{1}$ exists'',  
$\dot{\in}$ is always to be interpreted by $\in\res\mathcal{M}$, $\dot{I}_{\NS}$ is always to be interpreted by $\NS_{\omega_{1}}$ in the sense of $\mathcal{M}$. 
\item For a forcing $\IP$ and a $\IP$-name $\tau$, we denote $\tau^{G}$ by the interpretation of $\tau$ by $\IP$-generic $G$. 
\end{itemize}
\section*{Acknowledgement}
The second author was funded by Study Scholarships -Master Studies for All Academic Disciplines given by \textit{Deutscher Akademischer Austauschdienst} (DAAD, German Academic Exchange Service) during his master's study. 
The second author would like to thank DAAD for the generous support during his Master's study. 

\input{pre}
\input{mainthm}

\input{ground_model}
\bibliographystyle{plain}
\bibliography{set_theory}
\end{document}

%% file: makros.tex
\def\N{\text{\rom I\kern-.23em\rom N}}
\mathchardef\qin'1062
\mathchardef\qprec'1036
\mathchardef\qless'474
\mathchardef\sim'1030

\makeatletter
\def\moverlay{\mathpalette\mov@rlay}
\def\mov@rlay#1#2{\leavevmode\vtop{%
   \baselineskip\z@skip \lineskiplimit-\maxdimen
   \ialign{\hfil$\m@th#1##$\hfil\cr#2\crcr}}}
\newcommand{\charfusion}[3][\mathord]{
    #1{\ifx#1\mathop\vphantom{#2}\fi
        \mathpalette\mov@rlay{#2\cr#3}
      }
    \ifx#1\mathop\expandafter\displaylimits\fi}
\makeatother

\newtheorem{Satz von Cantor}[Corollary]{Satz (Cantor)}
\newtheorem{Satz von Cantor--Schroeder--Bernstein}[Corollary]{Satz von
  Cantor--Schrder--Bernstein}
\newtheorem{Satz von Cantor--Bendixson}[Corollary]{Satz von Cantor--Bendixson}

\newcommand{\ZFC}{{\sf ZFC}}

\newcommand{\Col}{\operatorname{Col}}

\newcommand{\Hom}{\operatorname{Hom}}

\newcommand{\ran}{\operatorname{ran}}

\newcommand{\dom}{\operatorname{dom}}

\newcommand{\On}{\operatorname{ON}}

\newcommand{\goed}[1]{\ulcorner #1 \urcorner}

\newcommand{\lh}{\operatorname{lh}}

\renewcommand{\N}{\mathbb{N}}

\newcommand{\R}{\mathbb{R}}
\newcommand{\IP}{\mathbb{P}}

\newcommand{\al}{\alpha}

\newcommand{\la}{\lambda}
\newcommand{\si}{\sigma}
\newcommand{\Si}{\Sigma}

\newcommand{\res}{\!\restriction\!}
\newlength{\oben}
\newlength{\unten}

\newenvironment{envara}[1]
        {\vspace{-1ex}\begin{list}{}{\settowidth{\labelwidth}{#1}
         \labelsep1em\parskip0pt\setlength{\leftmargin}{1.5em}
         \addtolength{\leftmargin}{\labelwidth}
         \parsep0pt\itemsep0pt\partopsep0pt
         }}{\end{list}}

\def\goed#1{\setbox5=\hbox{$#1$}\dimen1=.25em \dimen2=\dimen1 \advance \dimen2 
by -1pt
\dimen4=0pt%
\dimen3=\ht5%
\dimen5=\wd5%
\ifdim\dimen3=\dimen4\dimen3=1.5ex\fi%
\hbox{\hskip2pt plus 2pt minus 1pt\raise.65\dimen3 \hbox{%
\vrule height.5\dimen3 depth0pt  width.4pt%
\vrule  height.5\dimen3 width\dimen1 depth-.48\dimen3}%
\kern-\dimen2%
\ifdim\dimen5=0pt\hskip1em\else\copy5\fi\kern-\dimen2
\raise.65\dimen3 \hbox{\vrule height .5\dimen3 width\dimen1 depth-.48\dimen3
\vrule height.5\dimen3 depth 0pt width.4pt}\hskip2pt plus2pt minus1pt}}

\newcommand{\AD}{{\sf AD}}
\newcommand{\ZF}{{\sf ZF}}
\newcommand{\AC}{{\sf AC}}
\newcommand{\FA}{{\sf FA}}
\newcommand{\BMM}{{\sf BMM}}

\newcommand{\DC}{{\sf DC}}

\newcommand{\HC}{\mathrm{HC}}

\newcommand{\BT}{{\sf BT}}
\newcommand{\MM}{{\sf MM}}

\newcommand{\MA}{{\sf MA}}
\newcommand{\NS}{{\sf NS}}

\newcommand{\Lg}{{\langle}}
\newcommand{\Rg}{{\rangle}}

\newcommand{\TC}[1]{\mathrm{TC}(#1)}

%% file: Introduction.tex
\section{Introduction}
In retrospect, the program of obtaining Martin's Maximum${}^{++}$, abbreviated by $\MM^{++}$, or just consequences thereof by forcing over models of determinacy started with the work of Steel and Van Wesep \cite{SW}. 
They obtained the consistency of the saturation of $\NS_{\omega_{1}}$ plus $\underset{\sim}{\delta}^{1}_{2} = \omega_{2}$ by forcing over a model of $\AD_{\R}+``\Theta$ is regular.''
Later, Woodin introduced the partial order $\IP_{\max}$ \cite[Definition 4.33]{Wo} and forced a restricted version of $\MM^{++}$ called $\MM^{++}(\mathfrak{c})$ over models of $\AD_{\R}+``\Theta$ is regular'' \cite[Theorem 9.39]{Wo}, and he also forced Bounded Martin's Maximum${}^{++}$, abbreviated by $\BMM^{++}$, over models of $\AD+``V$ is closed under the $M_{1}^{\sharp}$ operator'' \cite[Theorem 10.99]{Wo}. Recently, Larson and Sargsyan \cite{LS} forced $\MM^{++}(\mathfrak{c})$ plus failures of square over Chang models by $\IP_{\max}$ forcing. Schindler \cite[Definition 2.10]{Sch} introduced $\MM^{*, ++}$ as a strengthening of $\MM^{++}$. $\MM^{*, ++}$ is defined by replacing the clause``$\varphi(\mathcal{M})$ may be forced to hold in stationary set preserving forcing extensions of $V$'' in the reformulations of $\MM^{++}$ as in \cite{BB} and \cite[Theorem 1.3]{BR2} with ``$\varphi(\mathcal{M})$ is honestly consistent.'' $\MM^{*, ++}$ is a natural statement in the context of $\IP_{\max}$ extensions, and the ultimate goal of the program seems to be to get models of $\MM^{*, ++}$ using $\IP_{\max}$ forcing. 

Our main result is part of this program and is on forcing $\MM^{*, ++}_{\mathfrak{c}}$. $\MM^{*, ++}_{\mathfrak{c}}$ is a global fragment of $\MM^{*, ++}$ stronger than both $\BMM^{++}$ and $\MM^{++}(\mathfrak{c})$. 

We let $\Gamma^{\infty}$ be the set of universally Baire sets of reals. We say a pointclass $\Gamma$ consisting of universally Baire sets of reals is \textit{productive} if it is closed under complements, projections, and satisfies for all $A\in\Gamma$, 
\[ \exists^{\R}A^{*}=(\exists^{\R}A)^{*}\]
holds in all generic extensions, where $A^{*}$ and $(\exists^{\R}A)^{*}$ denotes the canonical extension of $A$ and $\exists^{\R}A$ respectively (see Definition \ref{productivity}). 

Here is the first main result. 
\begin{mainthm}\label{mainthm}
Suppose $\Gamma\subset\mathcal{P}(\R)$ is a boldface pointclass and $F$ is a class predicate such that $L^{F}(\Gamma, \R)$ satisfies the following: 
\begin{enumerate}
    \item $L^{F}(\Gamma, \R)\cap \mathcal{P}(\R)=\Gamma$, 
    \item $\ZF+\AD^{+}+\AD_{\R}+``\Theta$ is regular'', 
    \item Every set of reals is universally Baire, 
    \item $\Gamma^{\infty}(=\mathcal{P}(\R))$ is productive. 
\end{enumerate}
    Suppose that $G\subset\IP_{\max}$ is $L^{F}(\Gamma, \R)$-generic. Suppose 
    \[H\subset\mathrm{Add}(\omega_{3}, 1)^{L^{F}(\Gamma, \R)}\]
    is $L^{F}(\Gamma, \R)[G]$-generic. Then 
    \[ L^{F}(\Gamma, \R)[G][H]\models \ZFC+\MM^{*, ++}_{\mathfrak{c}}. \]
\end{mainthm}

By combining the results in \cite[Lemma 3.4]{TG} and \cite[Main Theorem]{uB}, we construct a ground model for Main Theorem \ref{mainthm}. Here is the second main result. 
\begin{mainthm}\label{maindm}
        Suppose that $V$ is self-iterable\footnote{Broadly speaking, $V$ is self-iterable if it knows its unique iteration strategy in any set generic extensions. See \cite[Section 2]{TG}. }. Let $\lambda$ be an inaccessible cardinal which is a limit of Woodin cardinals and a limit of strong cardinals, and let $G\subset\Col(\omega, <\lambda)$ be a $V$-generic filter. 
        Let 
        \[\mathcal{M}=(L^{F_{\mathrm{uB}}}(\R^{*}, \Hom^{*}))^{V(\R^{*})}, \]
        where $F_{\mathrm{uB}}$ is as defined in \cite[Definition 4.4]{uB}\footnote{We shall give the definition of $F_{\mathrm{uB}}$ later. See Definition \ref{predF}. }. 
        Then
        \begin{enumerate}
            \item $\mathcal{M}\models\AD^{+}+\AD_{\R}+ ``\Theta \text{ is regular''} + \text{ ``Every set of reals is universally Baire''}$, 
            \item $\mathcal{M}\cap \mathcal{P}(\R^{*}_{G})=\Hom^{*}_{G}$, 
            \item $\mathcal{M}\models``\Gamma^{\infty}(=\Hom^{*}_{G})\text{ is productive''}$.  
        \end{enumerate}
    \end{mainthm}
From these two results, the theory     
\begin{align*}
    &\ZF+\AD^{+}+\AD_{\R}+``\Theta \text{ is regular''}+\\
    &``\text{Every set of reals is universally Baire''}+``\Gamma^{\infty} \text{ is productive''}. 
\end{align*}
seems a reasonable theory beyond $V=L(\mathcal{P}(\R))$. However, it is still open whether the theory 
\begin{align*}
    &\ZF+\AD^{+}+\AD_{\R}+``\Theta \text{ is regular''}+``\text{Every set of reals is universally Baire''}
\end{align*}
implies the productivity of $\Gamma^{\infty}$ or not. The similar question was asked by Feng--Magidor--Woodin \cite[6. OPEN QUESTIONS 3]{FMW}. 

Chapter $2$ lists preliminaries. In Chapter $3$, we prove the existence of capturing mice, which is crucial for the proof of the first main theorem. Chapter $4$ is the proof of the first main theorem, and Chapter 5 is the proof of the second main theorem.

%% file: pre.tex
\section{Preliminaries}
This chapter lists the tools and theorems used in this paper briefly. 
\subsection{Forcing axioms }
In this section, we introduce bounded forcing axioms and their characterization. 
\begin{defn}
    Let $\Gamma$ be a class of forcings, i.e., complete Boolean algebras, and let $\kappa$ be an uncountable cardinal. 
    
    $\FA_{\kappa}(\Gamma)$, or $\FA_{\kappa}$ for forcings in $\Gamma$, denotes the statement that whenever $\IP\in\Gamma$ and $\{A_{i} \mid i<\omega_{1}\}$ is a family of maximal antichains in $\IP$ such that $A_{i}$ has size at most $\kappa$ for each $i<\omega_{1}$, then there is a filter $G$ in $\IP$ such that $G\cap A_{i}\neq\emptyset$ for all $i<\omega_{1}$. 

    $\FA_{\kappa}^{++}(\Gamma)$, or $\FA_{\kappa}^{++}$ for forcings in $\Gamma$, denotes the statement that whenever $\IP\in\Gamma$, $\{A_{i} \mid i<\omega_{1}\}$ is a family of maximal antichains in $\IP$, and $\{\tau_{i} \mid i<\omega_{1}\}$ is a family of terms for stationary subsets of $\omega_{1}$ such that $A_{i}$ has size at most $\kappa$ for each $i<\omega_{1}$, then there is a filter $G$ in $\IP$ such that $G\cap A_{i}\neq\emptyset$ for all $i<\omega_{1}$ and 
    \[\tau_{i}^{G}=\{\alpha<\omega_{1} \mid \exists p\in G (p\dststile{\IP}{}\check{\alpha}\in\dot{\tau}_{i}) \} \]
    is stationary for all $i<\omega_{1}$.
\end{defn}
We have $\BMM$ is $\FA_{\aleph_{1}}$ for stationary set preserving forcings, $\MM_{\mathfrak{c}}$ is $\FA_{\mathfrak{c}}$ for stationary set preserving forcings, and $\MM$ is $\FA_{\kappa}$ for stationary set preserving forcings and for  all $\kappa$. 
The same goes for the $++$ version. 
\begin{defn}
    Let $\mathcal{M}=(M, \in, \vec{R})$ be a transitive structure such that $\vec{R}=(R_{i} \colon i<\omega_{1})$ is a list of $\aleph_{1}$ relations on $M$, and let $\varphi$ be a $\Si_{1}$ formula. Let $\Psi(\mathcal{M}, \varphi)$ be the statement that there is some transitive structure $\bar{\mathcal{M}}$ of size $\aleph_{1}$, some elementary $\pi\colon \bar{\mathcal{M}}=(\bar{M}, \in, (\bar{R}_{i} \colon i<\omega_{1}))\to \mathcal{M}$, and $\varphi(\bar{\mathcal{M}})$ holds true. 
\end{defn}
Honest consistency is motivated by the following characterization of $\FA^{++}_{\kappa}(\Gamma)$. 
\begin{lem}[{\cite[Theorem 5]{BB}}, {\cite[Theorem 1.3]{BR2}}]\label{char}
    Let $\Gamma$ be a class of forcings. The following are equivalent. 
    \begin{enumerate}
        \item $\FA^{++}_{\kappa}(\Gamma)$. 
        \item For all $\IP\in\Gamma$, for all transitive structures $\mathcal{M}$ of size at most $\kappa$, and for all $\Si_{1}$ formulae $\varphi$ in $\mathcal{L}_{\dot{\in}, \dot{I}_{\NS}}$, 
        \[ V^{\IP}\models\varphi(\mathcal{M})\Longrightarrow V\models\Psi(\mathcal{M}, \varphi). \]
    \end{enumerate}
\end{lem}
\begin{proof}
We only prove that (2) implies (1). For more details, see \cite[Theorem 5]{BB} and \cite[Theorem 1.3]{BR2}. 

Let $\IP$ be stationary set preserving complete Boolean algebra, let $(A_{i} \colon i<\omega_{1})$ be a family of maximal antichains in $\IP$ such that each $A_{i}$ has size at most $\kappa$, and let $(\tau_{i} \colon i<\omega_{1})$ be a faimily of names for stationary subsets of $\omega_{1}$. 
Let 
\[ B_{i}=\{ (\alpha, ||\check{\alpha}\in\dot{\tau}_{i}||) \mid \alpha<\omega_{1}\}\]
for $i<\omega_{1}$. 
Let $\theta$ be sufficiently large, and let 
\[ \si\colon\mathcal{M}=(M, \in, \bar{\IP}, (\bar{A}_{i} \colon i<\omega_{1}), (\bar{B}_{i} \colon i<\omega_{1}))\to(H_{\theta}, \in, \IP, (A_{i} \colon i<\omega_{1}), (B_{i} \colon i<\omega_{1})), \]
where $M$ is transitive and of size $\kappa$, and 
\[ (\omega_{1}+1)\cup\{\IP\}\cup \{\tau_{i}\mid i<\omega_{1}\}\cup \bigcup_{i<\omega_{1}}A_{i}\cup\bigcup_{i<\omega_{1}}B_{i}\subset\ran(\si). \]
Let $\varphi(\mathcal{M})$ say 
\begin{align*}
    \varphi(\mathcal{M})\equiv &\exists \tilde{G}\, \exists (S_{i} \mid i<\omega_{1}) \\
    &[\bar{G} \text{ is a filter in } \bar{\IP}\\
    &\land\forall i<\omega_{1} (\bar{A}_{i}\cap \bar{G}\neq\emptyset)\\
    &\land \forall i<\omega_{1} (S_{i}\subset\omega_{1} \land S_{i}\notin \dot{I}_{\NS}) \\
    &\land \forall i<\omega_{1}\forall \alpha\in S_{i}\exists p\in\bar{G} \\
    &((\al, p)\in\bar{B}_{i})]. 
\end{align*}
Let $G\subset\IP$ be $V$-generic. Then 
\[ V[G]\models \varphi(\mathcal{M})\]
as witnessed by the $\si$-preimage of $G$ and $(\tau_{i}^{G} \colon i<\omega_{1})$.  

By our hypothesis, let $\pi\colon\bar{\mathcal{M}}\to \mathcal{M}$ be elementary and such that $\varphi(\bar{\mathcal{M}})$ holds. 
Let $\bar{G}$ and $(S_{i}\colon i<\omega_{1})$ witness that $\varphi(\bar{\mathcal{M}})$ holds true. 
Let $G$ be the filter in $\IP$ generated by $(\si\circ\pi)[\bar{G}]$. 
Then we have that 
\[ G\cap A_{i}\neq\emptyset, \]
since $\bar{G}\cap\bar{A}_{i}\neq\emptyset$. 
Moreover, since $S_{i}\subset \tau_{i}^{\bar{G}}$ for each $i<\omega$, $\tau_{i}^{\bar{G}}$ is stationary for each $i<\omega_{1}$. 
Hence $G$ is the desired filter. Therefore, $\FA^{++}_{\kappa}(\Gamma)$ holds. 
\end{proof}

\subsection{Universally Baire property}
The notion of universally Baireness is introduced by Feng, Magidor, and Woodin \cite{FMW}. The following definition of universally Baireness is still valid in choiceless models. 
\begin{defn}
    Let $T$ and $U$ be trees on $\omega^{k}\times\On$, and let $Z$ be a set. We say the pair $(T, U)$ is \textit{$Z$-absolutely complementing} if we have
    \[ V^{\Col(\omega, Z)}\models p[T]=\R^{k}\setminus p[U]. \]

    We say a set of reals $A$ is \textit{$Z$-universally Baire} if there is a $Z$-absolutely complementing pair $(T, U)$ of trees such that $A=p[T]=\R\setminus p[U]$ in $V$. 
    We call such a pair $(T, U)$ \textit{$Z$-absolutely complementing pair of trees for $A$}. We say a set of reals $A$ is \textit{universally Baire} if $A$ is $Z$-universally Baire for any set $Z$. 
\end{defn}
In the $\AC$ context, being universally Baire is equivalent to being $<\On$-universally Baire, i.e., $\kappa$-universally Baire for all $\kappa\in\On$. However, in the absence of $\AC$, the authors do not know that they are still equivalent. 
\begin{notation}
Let $Z$ be a set. 
    Let $A$ be a $Z$-universally Baire, and let $(T, U)$ be a $Z$-absolutely complementing pair of trees for $A$. 
    Let $G$ be $\Col(\omega, Z)$-generic. 
    Then we denote the canonical expansion of $A$ to $V[G]$ by
    \[A^{G}=p[T]^{V[G]},  \]
    or if $G$ is clear from the context, 
    we denote it by $A^{*}$.

    Note that the canonical expansion does not depend on the choice of an absolutely complementing pair of trees. 
\end{notation}
We shall use projective generic absoluteness with names for sets of reals in the proof of Theorem \ref{mainthm}. For that, we need tree representations compatible with projections, i.e., $(\exists^{\R} A)^{*}=\exists^{\R}A^{*}$ for a universally Baire set $A$. 
\begin{defn}\label{productivity}
    Let $\Gamma\subset\bigcup_{1\leq k<\omega}\mathcal{P}(\R^{k})$ be a pointclass of universally Baire sets of reals. We say that $\Gamma$ is \textit{productive} if 
    \begin{enumerate}
        \item $\Gamma$ is closed under taking complements and projections, and 
        \item for all $k<\omega$ and for all $D\in \Gamma\cap \mathcal{P}(\R^{k+2})$, if the trees $T$ and $U$ on $\omega^{k+2}\times\On$ witness that $D$ is $Z$-universally Baire and if 
        \[ \tilde{T}=\{(s\res (k+1), (s(k+1), t)) \mid (s, t)\in T \}, \]
        then there is a tree $\tilde{U}$ on $\omega^{k+1}\times\On$ such that
        \[ V^{\Col(\omega, Z)}\models p[\tilde{T}]=\R^{k+1}\setminus p[\tilde{U}]. \]
    \end{enumerate}
\end{defn}
Lemma \ref{abs} is shown by an induction on the complexity of formulae. 
\begin{lem}\label{abs}
    Let $\Gamma$ be productive, and let $A\in \Gamma$. Then any projective statement about $A$ is absolute between $V$ and any forcing extension of $V$. 
\end{lem}
\subsection{\texorpdfstring{$\MM^{*, ++}$}{Lg}}
Let us introduce $\MM^{*, ++}$. The notations are mainly the same as \cite{Sch}. We say $x\in\R$ \textit{codes a transitive set} if 
\[ E_{x}=\{(n, m) \mid x(\Lg n, m\Rg)=0\}\]
is an extensional and well-founded relation on $\omega$. 
Let us write $\mathrm{WF}$ for the set of reals coding a transitive set. Note that $\mathrm{WF}$ is the $\Pi^{1}_{1}$-complete set of reals\footnote{See \cite[4A]{Mos}.}. 
If $x\in\mathrm{WF}$, then let $\pi_{x}$ be the transitive collapse of the structure $(\omega, E_{x})$ and let $\mathrm{decode}(x)=\pi_{x}(0)$. If $x, y\in\mathrm{WF}$, then we shall write $x\simeq y$ to express that $\mathrm{decode}(x)=\mathrm{decode}(y)$. 

We say a function $f\colon\R\to\R$ is \textit{universally Baire} if the graph of $f$ is a universally Baire subset of $\R^{2}$. 
\begin{defn}\label{def-stub}
  We say a function $F\colon \HC\to\HC$ is \textit{strongly universally Baire in the codes}, or briefly \textit{strongly universally Baire}, if there is a universally Baire function $f\colon\R\to\R$ such that 
  \begin{enumerate}
      \item if $z\in\HC$ and $x\in\mathrm{WF}$ with $z=\mathrm{decode}(x)$, then $f(x)\in\mathrm{WF}$ and $F(z)=\mathrm{decode}(f(x))$; 
      
      let $(T, U)$ witness that $f$ is universally Baire with $f=p[T]$. Then for all posets $\IP$, 
      \item $V^{\IP}\models ``p[T]$ is a function from $\R$ to $\R$''; 
      \item $V^{\IP}\models \forall \{x, x', y, y'\}\subset\R [(x, y), (x', y')\in p[T] \land x\simeq x'\longrightarrow y\simeq y']$. 
  \end{enumerate}
\end{defn}
Note that both (2) and (3) of Definition \ref{def-stub} are projective statements about the graph of $f$. Hence by Lemma \ref{abs}, we have the following. 
\begin{lem}\label{lem-stub}
Assume that $\Gamma^{\infty}$ is productive. Let $F\colon \HC\to\HC$ be a function. Suppose that there is a  universally Baire function $f\colon\R\to\R$ satisfying 
\begin{enumerate}
    \item if $z\in\HC$ and $x\in\mathrm{WF}$ with $z=\mathrm{decode}(x)$, then $f(x)\in\mathrm{WF}$ and $F(z)=\mathrm{decode}(f(x))$; and
    \item $\forall \{x, x', y, y'\}\subset\R [y=f(x) \land y'=f(x') \land x\simeq x'\longrightarrow y\simeq y']$. 
\end{enumerate}
Then $F$ is strongly universally Baire. 
\end{lem}
Let $F\colon \HC\to\HC$ be strongly universally Baire in the codes as witnessed by $f$ (and a pair $(T, U)$ of trees). Let $\IP$ be a poset, and let $g\subset\IP$ be $V$-generic. Then $F$ canonically extends a total map 
\[ F^{\IP, g}\colon V[g]\to V[g]\]
as follows. 
Let $X\in V[g]$. Let $\theta$ be any sufficiently large cardinal, let $H\subset\Col(\omega, \theta)$ be $V[g]$-generic, and let $x\in \R\cap V[g][H]$ be such that $X=\mathrm{decode}(x)$. Let $y\in\R\cap V[g][H]$ be such that $(x, y)\in p[T]$. We set $F^{\IP, g}(X)=\mathrm{decode}(y)$. Then $F^{\IP, g}$ is well-defined and does not depend on the choice of $f$, $(T, U )$, $\theta$, and $H$ by the standard argument. 

\begin{defn}
    Let $F\colon \HC\to\HC$ be strongly universally Baire. Let $\theta\in\On$, let $g\subset \Col(\omega, \theta)$ be $V$-generic, and let $\mathfrak{A}\in V[g]$ be transitive. We say $\mathfrak{A}$ is $F$-closed if 
    \begin{enumerate}
        \item $\mathfrak{A}$ is closed under $F^{\Col(\omega, \theta), g}$, and 
        \item $F^{\Col(\omega, \theta), g}\res X\in \mathfrak{A}$ for every $X\in\mathfrak{A}$. 
    \end{enumerate}
\end{defn}
\begin{defn}[{\cite[Definition 2.8]{Sch}}]
    Let $\varphi$ be a formula in the language $\mathcal{L}_{\dot{\in}, \dot{I}_{\NS}}$, and let $M\in V$. Let $\theta=\aleph_{1}+|\TC{\{M\}}|$. We say $\varphi(M)$ is \textit{honestly consistent} if for every $F\colon \HC\to\HC$ which is strongly universally Baire in the codes, if $g\subset\Col(\omega, 2^{\theta})$ is $V$-generic, then in $V[g]$ there is a transitive model $\mathfrak{A}$ such that
    \begin{enumerate}
        \item $\mathfrak{A}$ is $F$-closed, 
        \item $\mathfrak{A}\models\ZFC^{-}$, 
        \item $(H_{\theta^{+}})^V\in \mathfrak{A}$, 
        \item ${\dot{I}_{\NS}}^{\mathfrak{A}}\cap V=\NS_{\omega_{1}}^{V}$, and 
        \item $\mathfrak{A}\models\varphi(M)$. 
    \end{enumerate}
\end{defn}

\begin{defn}[{\cite[Definition 2.10]{Sch}}] Let $\kappa$ be an infinite cardinal. 
\begin{itemize}
    \item ${\textit{Martin's Maximum}}_{\kappa}^{*, ++}$, abbreviated by $\MM_{\kappa}^{*, ++}$, is the statement that whenever $\mathcal{M}=(M, \in, \vec{R})$ is a models, where $M$ is transitive with $|M|\leq\kappa$ and $\vec{R}$ is a list of $\aleph_{1}$ relations on $M$ and whenever $\varphi$ is a $\Si_{1}$ formula in the language $\mathcal{L}_{\dot{\in}, \dot{I}_{\NS}}$ such that $\varphi(\mathcal{M})$ is honestly consistent, then $\Psi(\mathcal{M}, \varphi)$ holds true in $V$.
    \item $\MM^{*, ++}$ is the statement that $\MM^{*, ++}_{\kappa}$ holds for all $\kappa$.  
\end{itemize}
\end{defn}
Asper\'o and Schindler \cite[Definition 2.3 and Theorem 3.1]{DR2} introduced a weaker notion, called \textit{$1$-honestly consistency}, and proved that $\MM$ is $\Si_{2}$-complete. 

The same argument as in Lemma \ref{char} shows that: 
\begin{thm}[Schindler, {\cite[Theorem 2.11]{Sch}}]\label{charchar}
Let $\kappa$ be an infinite cardinal. Then $\MM_{\kappa}^{*, ++}$ implies $\MM_{\kappa}^{++}$. 
\end{thm}
\subsection{\texorpdfstring{$\IP_{\max}$}{Lg}}
The forcing $\IP_{\max}$ was introduced by Woodin, see \cite[Chapter 4]{Wo}. 

To define $\IP_{\max}$, we need the notion of ``generic iteration''. We consider structures of the form $(M, \in, I, a)$, where $M$ is a countable transitive model of a sufficiently large fragment of $\ZFC$, $(M, I)$ is amenable, $a\subset\omega_{1}^{M}$, and $I$ is a $M$-normal ideal over $\omega_{1}^{M}$. 
For $\gamma\leq\omega_{1}$, we say 
\[\Lg \Lg (M_{\al}, \in, I_{\al}, a_{\al}) \mid \al\leq\gamma\Rg, \Lg j_{\al, \beta} \mid \al\leq\beta\leq \gamma\Rg, \Lg G_{\al}\mid \al<\gamma\Rg\Rg\]
is a generic iteration of $(M,\in, I, a)$ of length $\gamma$ if the following hold
\begin{enumerate}
    \item $(M_{0}, \in, I_{0}, a_{0})=(M, \in, I, a)$, 
    \item for $\al<\gamma$, $G_{\al}$ is a $\mathcal{P}(\omega_{1})^{M_{\al}}/I_{\al}$-generic over $M_{\al}$, $M_{\al+1}$ is the generic ultrapower of $M_{\al}$ by $G_{\al}$, and $j_{\al, \al+1}\colon (M_{\al}, \in, I_{\al}, a_{\al})\to(M_{\al+1}, \in, I_{\al+1}, a_{\al+1})$ is the corresponding generic elementary embedding, 
    \item for $\al\leq\beta\leq\delta$, $j_{\al, \delta}=j_{\beta, \delta}\circ j_{\al, \beta}$, 
    \item for $\beta$ is a nonzero limit ordinal $\leq\gamma$, then $\Lg M_{\beta}, \Lg j_{\al, \beta} \mid \al<\beta\Rg\Rg$ is the direct limit of $\Lg M_{\al}, j_{\al, \al'} \mid \al\leq\al'<\beta\Rg$. 
\end{enumerate}
We say good $(M, \in, I, a)$ is \textit{generically iterable}, or briefly \textit{iterable}, if for all $\al\leq\omega_{1}$, every generic iteration of $(M, \in, I, a)$ of length $\al$ is well-founded. 

\begin{defn}
    The partial order $\IP_{\max}$ consists of all pairs $(M, \in, I, a)$ such that 
    \begin{enumerate}
        \item $M$ is a countable transitive model of a sufficiently large fragment of $\ZFC$+$\MA_{\aleph_{1}}$; 
        \item $(M, I)$ is amenable; 
        \item $I$ is a $M$-normal ideal on $\omega_{1}^{M}$, and $a\in \mathcal{P}(\omega_{1})^{M}$; 
        \item $(M, I)$ is iterable; and  
        \item there is an $x\in\R^{M}$ such that $\omega_{1}^{M}=\omega_{1}^{L[a, x]}$. 
    \end{enumerate}
    $(M, \in, I, a)<_{\IP_{\max}}(N, \in, J, b)$ if $(N, \in, J, b)\in \HC^{M}$ and there is a generic iteration $j\colon (N, \in, J, b)\to (N', \in, J', b')$ in $M$ such that $b'=a$ and $I\cap N'=J'$. 

    We say $(M, \in, I)$ is a \textit{$\IP_{\max}$ precondition} if there is $a\subset\omega_{1}^{M}$ such that $(M, \in, I, a)\in\IP_{\max}$. Similarly, we define a generic iteration of a precondition $(M, \in, I)$. 
\end{defn}
{\bf Notation.} For a $\IP_{\max}$ condition $p=(M, \in, I, a)\in\IP_{\max}$, we often identify $p$ with its universe $M$ for notational simplicity. 

In the $\IP_{\max}$ analysis, it is important to see for each $A\subset\R$ there are 
densely many conditions that keep track of a name of $A$. It comes from Suslin representations. 
\begin{defn}
Let $A$ be a set of reals. 
We say a precondition $(M, \in, I)$ is \textit{$A$-iterable} if 
\begin{enumerate}
    \item $A\cap M\in M$, and 
    \item $j(A\cap M)=A\cap M'$ whenever $j\colon (M, \in, I)\to(M', \in, I')$ is an iteration of $(M, \in, I)$. 
\end{enumerate}
\end{defn}
Determinacy ensures that there are densely many $A$-iterable conditions for each $A\subset\R$. Here, $\AD^{+}$ is a technical variant of $\AD$. 

\begin{lem}[Woodin]
    Assume $\AD^{+}$. Let $A$ be a set of reals. Then there are densely many conditions $p=(M, \in, I, a)\in \IP_{\max}$ such that 
    \begin{enumerate}
    \item $(\HC^{M}, \in, A\cap M)\prec ( \HC, \in, A)$, 
    \item $(M, \in, I)$ is $A$-iterable, and 
    \item if $j\colon (M, \in, I)\to (M', \in, I')$ is any iteration of $(M, \in, I)$, then 
    \[ ( \HC^{M'}, \in, A\cap M')\prec ( \HC, \in, A). \]
    \end{enumerate}
\end{lem}
\begin{proof}
The key fact is the following. 
\begin{fact}[Woodin, {\cite[Theorem 7.1]{DM1}}, {\cite{AD+}}]\label{AD+fact}
    Assume $\AD^{+}$. Then 
    \begin{enumerate}
        \item The pointclass $\Si^{2}_{1}$ has the scale property, and 
        \item Every lightface $\Si^{2}_{1}$ collection of sets of reals has a lightface $\Delta^{2}_{1}$ member. 
    \end{enumerate}
\end{fact}
Now, suppose otherwise. Let $p_{0}=(M_{0}, \in, I_{0}, a_{0})\in\IP_{\max}$, and let $A\subset \R$ be a conterexample to the statement of the theorem. Then we may assume that $A$ is Suslin and co-Sudlin by Fact \ref{AD+fact}. Let $T$ and $U$ be trees projecting $A$ and its complement respectively. 

Since $\Delta^{2}_{1}(p_{0})$ is closed under complements, projections, and countable unions, so there exist a $\Delta^{2}_{1}(p_{0})$ set $B\subset\R\times\R$ such that whenever $F\colon\R\to\R$ uniformizes $B$ and $N$ is a transitive model of $\ZF$ closed under $F$, then 
\[ (\HC^{N}, \in, A\cap N)\prec (\HC, \in, A). \]
Again, by Fact \ref{AD+fact}, let $F\colon\R\to\R$ be $\Delta^{2}_{1}(p_{0})$ and uniformize $B$. Let $V$ and $W$ be trees projecting $F$ and its complement respectively. 

The next key fact is the following. 
\begin{fact}[{\cite[Theorem 5.4]{KWAD}}]\label{KoeWod}
    Assume $\AD$. Let $S$ be a set of ordinals. Then there exists a real $x$ such that for all reals $y$ with $x\in L[S, y]$, 
    \[ \mathrm{HOD}^{L[S, y]}_{S}\models\ZFC+\omega_{2}^{L[S, y]} \text{ is a Woodin cardinal. }\]
\end{fact}
By Fact \ref{KoeWod}, let $M$ be a transitive proper class model of $\ZFC$ such that 
\begin{enumerate}
\item $p_{0}, T, U, V, W\in M$, and 
\item there exists a countable ordinal $\delta$ such that $M\models``\delta$ is a Woodin cardinal''. 
\end{enumerate}
Let $\kappa<\la<\delta$ be such that $\kappa$ is measurable in $M$ and $\la$ is inaccessible in $M$. Let $g_{0}\subset\Col(\omega, <\kappa)$ be $M$-generic, and let $g_{1}$ be $M[g_{0}]$-generic for the standard c.c.c. poset to force $\MA$. Then in $M[g_{0}]$, an ideal dual to a fixed normal measure on $\kappa$ in $M$ generates a precipitous ideal in $M[g_{0}]$. Since c.c.c. forcings preserve precipitous ideals, it also generates a precipitous ideal in $M[g_{0}, g_{1}]$. 

Let $I$ be a precipitous ideal in $M[g_{0}, g_{1}]$. Then $(M_{\la}[g_{0}, g_{1}], \in, I)$ is iterable. 

Since $T\in M$, $A\cap M_{\la}[g_{0}, g_{1}]\in M_{\la}[g_{0}, g_{1}]$. 
Since $M_{\la}[g_{0}, g_{1}]$ is closed under $F$, we have that 
\[ (\HC^{M_{\la}[g_{0}, g_{1}]}, \in, A\cap M_{\la}[g_{0}, g_{1}])\prec(\HC, \in, A). \]
Fix an iteration 
\[j\colon (M_{\la}[g_{0}, g_{1}], \in, I)\to (M', \in, I').\]
Let 
\[j^{*}\colon (M[g_{0}, g_{1}], \in, I)\to (M^{*}, \in, I^{*})\]
be the lift of $j$. Then we have that 
\[ p[T]=p[j^{*}(T)], \text{ and } p[U]=p[j^{*}(U)],\]
and similarly, 
\[ p[V]=p[j^{*}(V)], \text{ and } p[W]=p[j^{*}(W)]. \]
Hence $A\cap M'\in M'$, and $M^{*}$ is closed under $F$. 
So we have that 
\[ (\HC^{M'}, \in, A\cap M')\prec(\HC, \in, A). \]
This shows that $A$ is not a counterexample to the statement of the theorem. This is a contradiction!
\end{proof}
We list the facts about $\IP_{\max}$ we need. Assume $\AD^{+}$.
\begin{itemize}
    \item $\IP_{\max}$ is $\si$-closed and homogeneous.  
    \item $\IP_{\max}$ forces the following: 
    \begin{itemize}
        \item $2^{\aleph_{0}}=\aleph_{2}$, 
        \item $\Theta^{V}=\omega_{3}$, if $\Theta^{V}$ is regular in $V$, 
        \item $\NS_{\omega_{1}}$ is saturated. 
    \end{itemize}
    \item Let $G$ be $\IP_{\max}$-generic. Define 
    \[A_{G}:=\bigcup\{a \mid \exists (M, I)\,  [(M, \in, I, a)\in G ]\}. \]
    Then for every $p=(M, \in, I, a)\in G$, there is a unique generic iteration of $p$ of length $\omega_{1}$ sending $a$ to $A_{G}$. 
    
    We let $\mathcal{P}(\omega_{1})_{G}$ be the set of all $B$ such that there is a $(M, \in, I, a)\in G$ and $b\in\mathcal{P}(\omega_{1})^{M}$ such that $j(b)=B$ where $j$ is the unique generic iteration of $(M, I)$ of length $\omega_{1}$ sending $a$ to $A_{G}$. Then $\mathcal{P}(\omega_{1})_{G}=\mathcal{P}(\omega_{1})^{V}$. 
\end{itemize}
For more details of the $\IP_{\max}$ forcing, see \cite{Lar} and \cite{Wo}. 

\section{Capturing mice and \texorpdfstring{$\Gamma$}{Lg}-Woodins}
The following theorem is due to Woodin. 
\begin{thm}[Woodin, {\cite[Theorem 7.14]{outline}}]\label{extalg}
    Let $\Si$ be am $(\omega_{1}+1)$-iteration strategy for $M$, and suppose that $\kappa<\delta$ are countable ordinals such that 
    \[M\models \ZF^{-}+\delta \text{ is Woodin }, \]
    then there is a $\mathbb{Q}\subset V^{M}_{\delta}$ such that 
    \begin{enumerate}
        \item $M\models\mathbb{Q}$ is a $\delta$-c.c. complete Boolean algebra, and 
        \item for any real $x$, there is a countable iteration tree $\mathcal{T}$ on $M$ according to $\Si$ based on the window $(\kappa, \delta)$ with the last model $\mathcal{M}^{\mathcal{T}}_{\al}$ such that $i^{\mathcal{T}}_{0, \al}$ exists and $x$ is $i^{\mathcal{T}}_{0, \al}(\mathbb{Q})$-generic over $\mathcal{M}^{\mathcal{T}}_{\al}$. 
    \end{enumerate}
    We call the poset $\mathbb{Q}$ Woodin's extender algebra of $M$ based on the window $(\kappa, \delta)$ and denote it by $\mathbb{B}^{M}_{(\kappa, \delta)}$. In the case that $\kappa=0$, we denote it by $\mathbb{B}^{M}_{\delta}$

    The process of iterating $M$ to make a real generic is called genericity iteration. 
\end{thm}
The following concept is due to Woodin. 
\begin{defn}[term capturing]
    Let $A\subset \R$. Let $M$ be a countable premouse, let $\delta\in M$, and assume that $M\models\ZFC^{-}+``\delta$ is a Woodin cardinal''. Let $\Si$ be an $(\omega_{1}, \omega_{1}+1)$-iteration strategy for $M$. Let $\tau\in M^{\Col(\omega, \delta)}$. 

    Then we say $(M, \delta, \tau, \Si)$ \textit{captures} $A$ if the following hold: 
    \begin{enumerate}
        \item $\Si$ satisfies hull condensation\footnote{See \cite[Definition 1.31]{Gri}. }, and branch condensation\footnote{See \cite[Definition 2.4]{Gri}. }, and is positional\footnote{See \cite[Definition 2.35]{Gri}. }; 
        \item if $\mathcal{T}$ is an iteration tree on $M$ of successor length $\theta+1<\omega_{1}$ which is according to $\Si$ such that the main branch $[0, \theta]_{\mathcal{T}}$ does not drop, if 
        \[ \pi^{\mathcal{T}}_{0, \theta}\colon M\to\mathcal{M}^{\mathcal{T}}_{\theta}\]
        is the iteration map, and if $g\in V$ is $\Col(\omega, \pi^{\mathcal{T}}_{0, \theta}(\delta))$-generic over $\mathcal{M}^{\mathcal{T}}_{\theta}$, then 
        \[ \pi^{\mathcal{T}}_{0, \theta}(\tau)^{g}=A\cap \mathcal{M}^{\mathcal{T}}_{\theta}[g]. \]
    \end{enumerate}
    We say that $(M, \Si)$ captures $A$ if there is $\tau, \delta\in M$ such that $(M, \delta, \tau, \Si)$ captures $A$. 
\end{defn}

\begin{defn}
    Let $A\subset\R$. We say $(M, \Si)$ \textit{strongly captures} $A$ if $(M, \Si)$ captures $A$ and for all $\kappa\geq\aleph_{1}$, there is some $(\kappa^{+}, \kappa^{+}+1)$-iteration strategy $\tilde{\Si}$ for $M$ extending $\Si$ and such that $\tilde{\Si}$ has hull condensation, and branch condensation, and is positional. 
\end{defn}

We will make use of $\Si$-mice to produce $A$-iterable conditions for $A\in\Gamma^{\infty}$. For more details for $\Si$-mice, see \cite{CMI}. 

    Let $M$ be a premouse, and let $\Si$ be an iteration strategy for $M$. 
    Let $X$ be a self-well-ordered transitive set with $M\in L_{1}(X)$. We say $\mathcal{N}$ is a \textit{$\Si$-premouse} over $X$ if $\mathcal{N}$ is a $J$-model of the form $J_{\al}[\vec{E}, \vec{S},X]$ satisfying the following; 
\begin{enumerate}
    \item $\vec{E}$ codes a sequence of extenders satisfying the usual axioms for fine extender sequences. (See \cite{FSIT}, \cite{CMI}, and \cite{outline}.)
    \item $\vec{S}$ codes a partial iteration strategy for $M$ as follows: 

    let $\gamma <\gamma+\delta\leq\al$ be such that $J_{\gamma}[\vec{E}, \vec{S}, X]\models\ZFC^{-}$ and $\gamma$ is the largest cardinal of $J_{\gamma+\delta}[\vec{E}, \vec{S}, X]$. Suppose that $\mathcal{T}\in J_{\gamma}[\vec{E}, \vec{S}, X]$ is $J_{\gamma}[\vec{E}, \vec{S}, X]$-least such that $\mathcal{T}$ is an iteration tree on $M$ of limit length, $\mathcal{T}$ is according to $\vec{S}\res\gamma$, but $(\vec{S}\res\gamma)(\mathcal{T})$ is undefined, Suppose also that $\delta=\lh(\mathcal{T})$, and $\delta$ does not have measurable cofinality in $J_{\gamma+\delta}[\vec{E}, \vec{S}, X]$. Then $\Si(\mathcal{T})$ is defined, and $\vec{S}(\gamma+\delta)$ is an amenable code for $(\mathcal{T}, \Si(\mathcal{T}))$. 
\end{enumerate}
Let $\mathcal{N}$ be a $\Si$-premouse over $X$, and let $\Gamma$ be an iteration strategy for $\mathcal{N}$. 
Then we say $\Gamma$ \textit{moves $\Si$ correctly} if every iterate $\mathcal{N}'$ of $\mathcal{N}$ according to $\Gamma$ is again a $\Si$-premouse over $X$. 

We say $\mathcal{N}$ is a \textit{$\Si$-mouse over $X$} if for every sufficiently elementary $\pi\colon\bar{\mathcal{N}}\to\mathcal{N}$ with $\bar{\mathcal{N}}$ being countable and transitive, there is some iteration strategy $\Gamma$ for $\bar{\mathcal{N}}$ which witnesses $\bar{\mathcal{N}}$ is $\omega_{1}+1$-iterable and which moves $\Si$ correctly. 

The next is the $\Si$-mouse version of $M^{\#}_{n}$. 
\begin{defn}
    Let $M$ be a countable premouse, and let $\Si$ be an iteration strategy for $N$. Let $X$ be a self-well-ordered transitive set with $N\in L_{1}(X)$, and let $n\in\omega$. Then we denote by 
    \[ M^{\#, \Si}_{n}(X)\]
    the unique $\Si$-mouse over $X$ which is sound above $X$, not $n$-small above $X$, and such that every proper initial segment is $n$-small above $X$, if it exists. 
\end{defn}
The next theorem is crucial for Main Theorem \ref{mainthm}. 
\begin{thm}[Schindler, {\cite[Theorem 3.14]{Sch}}]\label{GOD}
Let $A\in\Gamma^{\infty}$, and suppose that $(M, \delta, \tau, \Si)$ captures $A$. Let $X\in\HC$, and suppose that 
\[ N=M^{\#, \Si}_{2}(M, X)\]
exists. Let $\delta_{0}$ be the bottom Woodin cardinal of $N$, let $g_{0}\in V$ be $(\Col(\omega_{1}, <\delta_{0}))^{N}$-generic over $N$, and let $g_{1}\in V$ be $\mathbb{Q}$-generic over $N[g_{0}]$, where $\mathbb{Q}\in N[g_{0}]$ is the standard c.c.c. forcing for Martin's Axiom. 
Let $\kappa$ be the critical point of the top extender of $N$. 
Then 
\[ p=((N||\kappa)[g_{0}, g_{1}], \in, (\NS_{\omega_{1}})^{(N||\kappa)[g_{0}, g_{1}]})\]
is an $A$-iterable $\IP_{\max}$ precondition.     
\end{thm}
In order to make use of  Theorem \ref{GOD}, we need to find capturing mice. 
One key idea is to capture by using self-justifying systems. 

We let $T_{0}$ denote the theory 
\begin{align*}
    &\ZF+\AD^{+}+\AD_{\R}+\\
    &``\text{Every set of reals is universally Baire''}+``\Gamma^{\infty} \text{ is productive''}. 
\end{align*}
Assuming $T_{0}$, we shall prove the following. 
\begin{enumerate}
    \item For every set of reals $A$, there is a boldface fine structural hybrid mouse $(\mathcal{M}, \Lambda)$ such that $(\mathcal{M}, \Lambda)$ strongly captures $A$. 
    \item For every set of reals $A$, letting $(\mathcal{M}, \Lambda)$ be as in (1), then for every $n\in\omega$, $M^{\sharp, \Lambda}_{n}$ operator is total and strongly universally Baire. 
\end{enumerate}

We shall make use of $\Gamma$-Woodins and shall follow the notations and the results in \cite{ADR}. To make this paper reasonably self-contained, let us recall definitions and results in \cite{ADR} briefly. For more details, see \cite{ADR}. 

We say a countable collection of sets of reals $\mathcal{A}\in\mathcal{P}(\R)^{\omega}$ is a \textit{self justifying system}, or briefly \textit{sjs}, if $\mathcal{A}$ is closed under taking complements and for every $A\in\mathcal{A}$ admits a scale $(\leq_{n} \colon n\in\omega)$ such that $\leq_{n}$ belongs to $\mathcal{A}$. 

Note that $\AD_{\R}$ implies that for every set of reals $A\subset\R$, there is a sjs $\mathcal{A}$ containing $A$. 

We say $\Gamma$ is a \textit{good pointclass} if 
\begin{enumerate}
    \item $\Gamma$ is closed under recursive substitution, $\forall^{\omega}$, $\exists^{\omega}$, and $\exists^{\R}$, 
    \item $\Gamma$ is $\omega$-parametrized and scaled. 
\end{enumerate}
Every good pointclass $\Gamma$ has its associated $C_{\Gamma}$-operator. For $x\in \R$, 
\[ C_{\Gamma}(x)=\{ y\in\R \mid \exists\xi<\omega_{1}\, \text{$y$ is $\Gamma(x, z)$ for all $z$ coding $\xi$} \}. \]
We extend $x\mapsto C_{\Gamma}(x)$ to countable transitive sets. For countable transitive $a$ and $x\in\R$, we say $x$ codes $a$ if $\varphi\colon (a, \in)\simeq (\omega, R)$ where $nRm$ if and only if $(n, m)\in x$. For $b\subset a$, let $b_{x}$ be the real $\varphi"b$. 

For countable transitive $a$, $C_{\Gamma}(a)$ denotes the set of $b\subset a$ such that $b_{x}\in C_{\Gamma}(x)$ for all $x$ coding $a$. 

We have a nice description of $C_{\Gamma}(a)$. 
\begin{thm}[Harrington--Kechris, {\cite{HaKe}}]
Assume $\AD$. Let $\Gamma$ be a good pointclass, and let $T$ be a tree of a $\Gamma$-scale on a $\Gamma$-universal set. Then for any countable transitive $a$, 
\[ C_{\Gamma}(a)=\mathcal{P}(a)\cap L(T\cup\{a\}, a). \]
\end{thm}

\begin{defn}[{\cite[Definition 1.4]{ADR}}]
    We say $\Gamma$ is a \textit{very good pointclass}, or briefly \textit{vg-pointclass}, if there is a sjs $\mathcal{A}$, $\gamma<\Theta^{L(\mathcal{A}, \R)}$, a $\Si_{1}$-formula $\varphi$, and a real $x$ such that $L_{\gamma}(\mathcal{A}, \R)$ is the least initial segment of $L(\mathcal{A}, \R)$ that satisfies 
\[ \ZF-{\sf P}+``\Theta \text{ exists''}+``V=L_{\Theta^{+}}(C, \R) \text{ for some $C\subset\R$''} +\varphi(x)\]
and $\Gamma=(\Si^{2}_{1}(\mathcal{A}))^{L_{\gamma}(\mathcal{A}, \R)}$. We let $M_{\Gamma}$ denote $L_{\gamma}(\mathcal{A}, \R)$. 
\end{defn}
Note that every vg-pointclass is good by scale analysis (see \cite{Scale}). 

We say a transitive model $P$ of $\ZFC-${\sf Replacement} is a \textit{$\Gamma$-Woodin} if for some $\delta$, 
\begin{enumerate}
    \item $P\models``\delta$ is the unique Woodin cardinal'', 
    \item $P=C_{\Gamma}^{\omega}(P)(:=\bigcup_{k<\omega}C^{k}_{\Gamma}(P))$, and 
    \item for every $P$-inaccessible cardinal $\eta<\delta$, 
    \[ C_{\Gamma}(V_{\eta}^{P})\models ``\eta \text{ is not a Woodin cardinal''. }\]
\end{enumerate}
Let $\Gamma$ be a vg-pointclass as witnessed by $\mathcal{A}$ and $M_{\Gamma}$. Let $B\in M_{\Gamma}\cap \mathcal{P}(\R)$ be $\mathrm{OD}^{M_{\Gamma}}(\mathcal{A})$, and let $a\in\HC$ be transitive. Define the canonical term relation $\tau^{a}_{B}$ consisting of pairs $(p, \si)$ such that 
\begin{enumerate}
    \item $p\in\Col(\omega, a)$, 
    \item $\si\in C_{\Gamma}(a)$ is a standard $\Col(\omega, a)$-name for a real, and 
    \item for a co-meager many $g\subset\Col(\omega, a)$, if $p\in g$, then $\si^{g}\in B$. 
\end{enumerate}
Note that $\tau^{a}_{B}\in C_{\Gamma}(C_{\Gamma}(a))$. We iterate this process. 
For $k=0$, let $\tau^{a}_{B, 0}=\tau^{a}_{B}$, and let $\tau^{a}_{B, k}=\tau^{C^{k}_{\Gamma}(a)}_{B, 0}$. 

For a $\Gamma$-Woodin $P$ and for $B\in \mathrm{OD}^{M_{\Gamma}}(\mathcal{A})$, we let $\tau^{P}_{B, k}=\tau^{V^{P}_{\delta^{P}}}_{B, k}$. 
\begin{thm}[{\cite[Theorem 1.7]{ADR}}]
    Assume $\AD^{+}$ and suppose that $\Gamma$ is a vg-pointclass. Let $\mathcal{A}$ and $M_{\Gamma}$ witness that $\Gamma$ is very good. Let $A\in\mathrm{OD}(\mathcal{A})^{M_{\Gamma}}$. 
    Then there is a pair $(P, \Si)$ and a $\Gamma$-condensing sequence $\vec{B}$\footnote{See \cite{ADR}. } such that 
    \begin{enumerate}
        \item $P$ is $\Gamma$-Woodin, 
        \item $\Si$ is a $\Gamma$-fullness preserving $(\omega_{1}, \omega_{1})$-iteration strategy for $P$, i.e., every $\Si$-iterate of $P$ is again $\Gamma$-Woodin,  
        \item for each $i\in\omega$, $\Si$ respects $B_{i}$, i.e., if whenever $i\colon P\to Q$ is according to $\Si$, then $i(\tau^{P}_{B_{i}, k})=\tau^{Q}_{B_{i}, k}$ for every $k$, 
        \item $\Si$ respects $A$, 
        \item for every $\Si$-iterate $Q$ of $P$, for every $i\in\omega$, and for every $Q$-generic $g\subset\Col(\omega, \delta^{Q})$, $(\tau_{B_{i}}^{Q})^{g}=Q[g]\cap B_{i}$, 
        \item for any tree $\mathcal{T}\in\dom(\Si)$, $\Si(\mathcal{T})=b$ if and only if either 
        \begin{enumerate}
            \item[(a)] $C_{\Gamma}(\mathcal{M}(\mathcal{T}))\models``\delta(\mathcal{T}) \text{ is not a Woodin cardinal''}$ and $b$ is the unique well-founded cofinal branch $c$ of $\mathcal{T}$ such that $C_{\Gamma}(\mathcal{M}(\mathcal{T}))\in \mathcal{M}^{\mathcal{T}}_{c}$, or 
            \item[(b)] $C_{\Gamma}(\mathcal{M}(\mathcal{T}))\models``\delta(\mathcal{T}) \text{ is a Woodin cardinal''}$ and $b$ is the unique well-founded cofinal branch $c$ of $\mathcal{T}$ such that letting $Q=C_{\gamma}^{\omega}(\mathcal{M}(\mathcal{T}))$, $\mathcal{M}^{\mathcal{T}}_{c}=Q$ and for every $i\in\omega$, $\pi^{\mathcal{T}}_{c}(\tau^{P}_{B_{i}})=\tau^{Q}_{B_{i}}$. 
        \end{enumerate}
    \end{enumerate}
\end{thm}
For the definition of $\Gamma$-condensing sequence, see \cite{ADR}. The point is that the sjs $\mathcal{A}$ is coded into a set appearing in $\vec{B}$. We call the pair $(P, \Si)$ above is a \textit{$\Gamma$-excellent pair}. 

Now assume $T_{0}$ for the rest of this section. We are given any set of reals $A\subset\R$. 

Let $\Gamma$ be a very good pointclass as witnessed by $\mathcal{A}$ and $M_{\Gamma}$ such that $A\in\mathcal{A}$. We may find such a $\Gamma$ by $\AD_{\R}$. 
Let $(P, \Si)$ be a $\Gamma$-excellent pair. 
Let $\vec{B}$ witness the $\Gamma$-excellence of $(P, \Si)$ and let $u\in\R$ code the sequence $(\tau^{P}_{B_{i}}\colon i<\omega)$. 

Then we have $M^{\#, \Si}_{n}(x)$ exists for every $n\in\omega$ and for every swo $x\in\HC$ (see \cite[Theorem 3.1]{ADR}). 

Let $\mathcal{M}=M^{\#, \Si}_{1}(u)$. Let $\Lambda$ be the $(\omega_{1}, \omega_{1})$-iteration strategy for $\mathcal{M}$. Then we have 

\begin{thm}[Sargsyan, {\cite[Theorem 4.1]{ADR}}]\label{capture}
Let $\delta$ be the Woodin cardinal of $\mathcal{M}$. Then there are trees $(T, S)\in\mathcal{M}$ on $\omega\times(\delta^{+})^{\mathcal{M}}$ such that $\mathcal{M}\models$``$(T, S)$ are $\delta$-complementing'' and whenever $i\colon\mathcal{M}\to\mathcal{N}$ is an iteration according to $\Lambda$ and $g\subset \Col(\omega, i(\delta))$ is $\mathcal{N}$-generic, $\mathcal{N}[g]\cap p[i(T)]=\mathrm{Code}(\Si)\cap \mathcal{N}[g]$. 
\end{thm}

We claim that $\mathcal{M}$ captures $A$. 
Let $(T, S)$ be as in Theorem \ref{capture}. 
Let $\delta$ be the Woodin cardinal of $\mathcal{M}$. 
We define a term $\tau\in \mathcal{M}^{\Col(\omega, \delta)}$ as follows. 
$(p, \si)\in\tau$ if 
\begin{enumerate}
    \item $p\in\Col(\omega, \delta)$, 
    \item $\si\in \mathcal{M}^{\Col(\omega, \delta)}$ is a standard name for a real, and 
    \item $p\dststile{\Col(\omega, \delta)}{\mathcal{M}}$ ``there is an iteration tree $\mathcal{U}$ on $P$ according to $p[\check{T}]$ with the last model $Q$ such that $\dot{\si}$ is $Q$-generic over the extender algebra at $\pi^{\mathcal{Q}}(\check{\delta^{P}})$ and 
    \[Q[\dot{\si}]\models \dststile{\Col(\omega,  \delta^{Q})}{} \dot{\si} \in \pi^{\mathcal{U}}(\check{\tau^{P}_{A}})\text{''}. \]
\end{enumerate}
Then it is clear that $(\mathcal{M}, \delta, \tau, \Lambda)$ captures $A$\footnote{The good properties of $\Lambda$ follow from the branch uniqueness.}.  
Moreover, by Lemma \ref{abs}, $(\mathcal{M}, \delta, \tau, \Lambda)$ strongly captures $A$. 

Then the argument of \cite[Theorem 3.1]{ADR} shows for every $n\in\omega$ and for every swo $x\in\HC$, $M^{\#, \Lambda}_{n}(x)$ exists. By Lemma \ref{lem-stub}, the operator $x\mapsto M^{\#, \Lambda}_{n}(x)$ is strongly universally Baire. Here, since $\rho_{\omega}(M_{n}^{\#, \Lambda}(x))=\omega$ for each $x$, we may  fix a canonical coding of each $M_{n}^{\#, \Lambda}(x)$ into a real.  

Hence we obtained that 
\begin{thm}\label{capture2}
    Assume $T_{0}$. Then for every $A$, there is a pair $(M, \Si)$ such that 
    \begin{enumerate}
        \item $M$ is a boldface fine structural hybrid premouse, 
        \item $(M, \Si)$ strongly captures $A$, and 
        \item for every $n\in\omega$, the operator $x\mapsto M^{\#, \Si}_{n}(x)$ is total and strongly universally Baire. 
    \end{enumerate}
\end{thm}

%% file: mainthm.tex
\section{\texorpdfstring{${\sf MM}^{*, ++}_{\mathfrak{c}}$}{Lg} in \texorpdfstring{$\mathbb{P}_{\max}$}{Lg} extensions}
In this section, we prove the first main theorem.  
\begin{thm}\label{mainthm2}
Suppose $\Gamma\subset\mathcal{P}(\R)$ is a boldface pointclass and $F$ is a class predicate such that $L^{F}(\Gamma, \R)$ satisfies the following: 
\begin{enumerate}
    \item $L^{F}(\Gamma, \R)\cap \mathcal{P}(\R)=\Gamma$, 
    \item $\ZF+\AD^{+}+\AD_{\R}+``\Theta$ is regular'', 
    \item Every set of reals is universally Baire, 
    \item $\Gamma^{\infty}(=\mathcal{P}(\R))$ is productive. 
\end{enumerate}
    Suppose that $G\subset\IP_{\max}$ is $L^{F}(\Gamma, \R)$-generic. Suppose 
    \[H\subset\mathrm{Add}(\omega_{3}, 1)^{L^{F}(\Gamma, \R)}\]
    is $L^{F}(\Gamma, \R)[G]$-generic. Then 
    \[ L^{F}(\Gamma, \R)[G][H]\models \ZFC+\MM^{*, ++}_{\mathfrak{c}}. \]
\end{thm}
\begin{proof} 
Let us assume that $V=L^{F}(\Gamma, \R)$. Note that $V\models T_{0}$. The same argument as \cite[Theorem 9.35]{Wo} gives us that 
\[ L^{F}(\Gamma, \R)^{\IP_{\max}}\models \omega_{2}{-}\DC\footnote{We will not use the notation $\DC_{\omega_{2}}$, since it is inconsistent with the notation $\DC_{\R}$.}. \]
Hence we have 
\[ L^{F}(\Gamma, \R)^{\IP_{\max}\ast\mathrm{Add}(\omega_{3}, 1)}\models \ZFC. \]

Let us fix a standard coding of $\IP_{\max}$-conditions into reals and identify a $\IP_{\max}$-condition and its code. 
For a filter $h\subset\IP_{\max}$, we define 
\[\R_{h}=\bigcup_{p\in h}\R\cap p. \] 
For $V$-generic $G\subset\IP_{\max}$ and $V[G]$-generic $H\subset\mathrm{Add}(\omega_{3}, 1)$, 
we note that 
\[\R_{G}=\R\cap V=\R\cap V[G]=\R\cap V[G][H], \]
since $\IP_{\max}$ is $\sigma$-closed and $\mathrm{Add}(\omega_{3}, 1)$ is $\omega_{3}$-closed. 
Also, note that 
\[ \mathcal{P}(\R)\cap V[G]=\mathcal{P}(\R)\cap V[G][H], \]
since $\mathrm{Add}(\omega_{3}, 1)$ is $\omega_{3}$-closed and $2^{\aleph_{0}}=\aleph_{2}$ in $V[G]$. 
Now suppose that 
\[ V^{\IP_{\max}\ast\mathrm{Add}(\omega_{3}, 1)}\models\lnot\MM^{*, ++}_{\mathfrak{c}}. \]
Assume that we have the following. 
\begin{enumerate}
    \item $\varphi$ is a $\Si_{1}$ formula; 
    \item $\dot{\mathcal{M}}=(\dot{M}, \in, \dot{\vec{R}})$ is a $\IP_{\max}$-name for a transitive structure with $\aleph_{1}$-many relations, and $\dot{f}\in V^{\IP_{\max}}$; 
    \item $(p, \dot{s})\in\IP_{\max}\ast\mathrm{Add}(\omega_{3}, 1)$; 
    \item $p\dststile{\IP_{\max}}{} \dot{f}\colon\R\to\dot{M}$ is bijective; and 
    \item $(p, \dot{s})\dststile{\IP_{\max}\ast\mathrm{Add}(\omega_{3}, 1)}{} \varphi(\dot{\mathcal{M}})$ is honestly consistent; and 
    \item $(p, \dot{s})\dststile{\IP_{\max}\ast\mathrm{Add}(\omega_{3}, 1)}{}\lnot\Psi(\dot{\mathcal{M}}, \varphi)$. 
\end{enumerate}
We shall derive a contradiction by finding $q<_{\IP_{\max}} p$ such that 
\[ (q, \dot{s})\dststile{\IP_{\max}\ast\mathrm{Add}(\omega_{3}, 1)}{}\Psi(\dot{\mathcal{M}}, \varphi). \]
Let 
\[ B_{=}=\{(p, x, y)\in \R^{3} \mid p\in \IP_{\max} \land p\dststile{\IP_{\max}}{}\dot{f}(x)=\dot{f}(y)\}, \]
let 
\[ B_{\in}=\{(p, x, y)\in \R^{3} \mid p\in \IP_{\max} \land p\dststile{\IP_{\max}}{}\dot{f}(x)\in\dot{f}(y)\}, \]
and let
\[ B_{\vec{R}}=\{(p, x, y)\in \R^{3} \mid p\in \IP_{\max} \land x\in\mathrm{WO} \land 
p\dststile{\IP_{\max}}{}\dot{f}(y)\in R_{|x|_{\mathrm{WO}}}\}. \]
For each formula and for each $\vec{z}\in\R^{<\omega}$, let $E_{\psi, \vec{z}}$ be a set of conditions $p\in\IP_{\max}$ such that
\begin{enumerate}
    \item $p$ decides the sentence $``\dot{\mathcal{M}}\models \exists x\psi(x, \dot{f}(\check{\vec{z}}))$'', and 
    \item if $p\dststile{\IP_{\max}}{}\dot{\mathcal{M}}\models\exists x\psi(x, \dot{f}(\check{\vec{z}}))$, then there is $x\in p$ such that 
    \[p\dststile{\IP_{\max}}{} \dot{\mathcal{M}}\models \psi(\dot{f}(\check{x}), \dot{f}(\check{\vec{z}})). \]
\end{enumerate}
Note that $E_{\psi, \vec{z}}$ is also a dense subset of $\IP_{\max}$. 
Let 
\[E=\{(p, \ulcorner\psi\urcorner, \vec{z}) \mid p\in E_{\psi, \vec{z}}\}. \]
Suppose a pair $(q, h)$ satisfies: 
\begin{enumerate}
    \item $q=(N, \in, J, a)\in \IP_{\max}$, $h\in q$ is a filter in $\IP_{\max}$; 
    \item $(N, \in)\models\ZFC^{-}$
    \item $q<_{\IP_{\max}} p$ for all $p\in h$; and 
    \item $q$ is $(B_{=}\oplus B_{\in}\oplus B_{\vec{R}}\oplus E)$-iterable. 
\end{enumerate}
We define a structure $(\R_{h}/\sim, \tilde{\in}, \tilde{R})$ as follows. 
For $x, y\in\R_{h}$ and for $\al<\omega_{1}^{q}$, we define 
\[x\sim y\iff \exists r \in h\, (r, x, y)\in B_{=}, \]
and 
\[ x\tilde{\in}y \iff \exists r\in h\, (r, x, y)\in B_{\in}, \]
and 
\[\tilde{R}_{\al}(x) \iff \exists r\in h\, \exists w\in \mathrm{WO}_{\al}\, (r, w, x)\in B_{\vec{R}}. \]
Since $q$ is $(B_{=}\oplus B_{\in}\oplus B_{\vec{R}}\oplus E)$-iterable and $h$ is a filter, the quotient $(\R_{h}/\sim, \tilde{\in}, \tilde{R})$ is well-defined. 
Note that the relation $\tilde{\in}$ is extensional and well-founded since $h$ is a filter and $\IP_{\max}$ is $\si$-closed. 
Then we let
\[ \si_{h}\colon \mathcal{M}_{h}:=(M_{h}, \in, \vec{R}_{h})\simeq (\R_{h}/\sim, \tilde{\in}, \tilde{R})\]
be the uncollapsing map. 

Now we say a pair $(q, h)$ is \textit{good} if $(q, h)$ satisfies: 
\begin{enumerate}
    \item $q=(N, \in, J, a)\in \IP_{\max}$, $h\in q$ is a filter in $\IP_{\max}$; 
    \item $(N, \in)\models\ZFC^{-}$
    \item $q<_{\IP_{\max}} p$ for all $p\in h$; 
    \item $q$ is $(B_{=}\oplus B_{\in}\oplus B_{\vec{R}}\oplus E)$-iterable; 
    \item $q\models \varphi(\mathcal{M}_{h})$; and 
    \item for every formula $\psi$ and for every $\vec{z}\in (\R_{h})^{<\omega}$, $E_{\psi, \vec{z}}\cap h\neq\emptyset$. 
\end{enumerate}
Note that ``being a good pair'' is $\Si^{1}_{2}$ statement in $B_{=}$, $B_{\in}$, $B_{\vec{R}}$, and $E$. 
We shall find a good pair in the generic extension of $V$ by $\IP_{\max}\ast\mathrm{Add}(\omega_{3}, 1)\ast\Col(\omega, 2^{\omega_{2}})$, and using Lemma \ref{abs}, we shall find a good pair in $V$. 

{\bf Claim.} Let $G\subset\IP_{\max}$ be $V$-generic with $p\in G$, let $H\subset\mathrm{Add}(\omega_{3}, 1)$ be $V[G]$-generic with $\dot{s}_{G}\in H$, and let $g\subset \Col(\omega, 2^{\omega_{2}})$ be $V[G][H]$-generic. Then, in $V[G][H][g]$ there is $q\in \IP_{\max}$ such that $(q, G)$ is good. 

{\bf Proof.} By Theorem \ref{capture2}, let $(N, \delta, \tau, \Si)$ be such that 
\begin{enumerate}
    \item $(N, \delta, \tau, \Si)$ strongly captures $(B_{=}\oplus B_{\in}\oplus B_{\vec{R}}\oplus E)$, and 
    \item the function $F\colon X\mapsto M^{\#, \Si}_{2}(X)$, where $X\in \HC$ is self-well-ordered and $N\in L_{1}(X)$, is well-defined, total, and strongly universally Baire. 
\end{enumerate}
Let $\mathfrak{A}\in V[G][H][g]$ be an $F$-closed witness to the fact that $\varphi(\dot{\mathcal{M}}^{G})$ is honestly consistent. Let $X\in\mathfrak{A}$ be transitive and such that 
\begin{enumerate}
    \item $(\mathcal{P}(\omega_{1})\cap \mathfrak{A})\cup\{(\NS_{\omega_{1}})^{\mathfrak{A}}, G, \dot{\mathcal{M}}^{G}\}\in X$, and 
    \item $X\models \varphi(\dot{\mathcal{M}}^{G})$. 
\end{enumerate}
Let $M=M^{\#, \Si}_{2}(X)$, and let $\delta_{0}$, $g_{0}$, $g_{1}$, $\mathbb{Q}$, and $\kappa$ be as in the statement of Theorem \ref{GOD}. 

Then by Theorem \ref{GOD}, inside $V[G][H][g]$ we have  
\[ q=((M||\kappa)[g_{0}, g_{1}], \in, (\NS_{\omega_{1}})^{(M||\kappa)[g_{0}, g_{1}]}, A_{G})\]
is such that 
\begin{enumerate}
    \item $q\in\IP_{\max}$ and $G\in q$; 
    \item $((M||\kappa)[g_{0}, g_{1}], \in)\models\ZFC$; 
    \item $q<_{\IP_{\max}}r$ for all $r\in G$; and 
    \item $q$ is $(B_{=}\oplus B_{\in}\oplus B_{\vec{R}}\oplus E)^{*}$-iterable. 
\end{enumerate}
Note that $\mathcal{M}_{G}=\dot{\mathcal{M}}^{G}$ where $\mathcal{M}_{G}$ is the transitive collapse of $(\R_{G}/\sim, \tilde{\in}, \tilde{R}_{G})$. Hence we have 
\[ q\models \varphi(\mathcal{M}_{G}). \]
Also, we have 
\[ E_{\psi, \vec{z}}\cap G\neq\emptyset\]
for every formula $\psi$ and for every $\vec{z}\in (\R_{G})^{<\omega}$ by the genericity of $G$. Therefore, $(q, G)$ is a good pair. {\bf (Q.E.D. Claim.)}

Therefore, by Lemma \ref{abs}, there is a good pair $(q, h)$ in $V$ with $q<_{\IP_{\max}}p$. We claim that 
\[ (q, \dot{s})\dststile{\IP_{\max}\ast\mathrm{Add}(\omega_{3}, 1)}{}\Psi(\dot{\mathcal{M}}, \varphi). \]
Let $G\subset\IP_{\max}$ be $V$-generic with $q\in G$, let $H\subset\mathrm{Add}(\omega_{3}, 1)$ be $V[G]$-generic with $\dot{s}^{G}\in H$. 
Let $j\colon q\to q^{*}$ is the generic iteration of length $\omega_{1}$ induced by $G$, and let $h^{*}=j(h)$. 

Since $q\models \varphi(\mathcal{M}_{h})$, we have $q^{*}\models \varphi(\mathcal{M}_{h^{*}})$. 
By the upward absoluteness of $\Si_{1}$ formulae, $\varphi(\mathcal{M}_{h^{*}})$ holds in $V[G]$. 

We define $\pi\colon\mathcal{M}_{h^{*}}\to \dot{\mathcal{M}}^{G}$ as follows. For $x\in \mathcal{M}_{h^{*}}$, letting $y\in \R_{h^{*}}$ be such that $\si_{h^{*}}(x)=[y]_{\sim}$ where $\si_{h^{*}}\colon \mathcal{M}_{h^{*}}\simeq (\R_{h^{*}}/\sim, \tilde{\in}, \tilde{R})$ is uncollapsing and ${[y]}_{\sim}$ is the $\sim$-equivalence class of $y$, we define $\pi(x)=\dot{f}^{G}(y)$. $\pi$ is well-defined, since $h^{*}\subset G$ and $q$ is $(B_{=}\oplus B_{\in}\oplus B_{\vec{R}}\oplus E)$-iterable. We shall show that $\pi$ is elementary by Tarski--Vaught test. 
 
Let $\vec{z}\in (\R_{h^{*}})^{<\omega}$ and let $\psi$ be a formula, and suppose that
\[ \dot{\mathcal{M}}_{G} \models \exists x \psi(x, \dot{f}^{G}(\vec{z})). \]
Since $(q, h)$ is good and $j$ is elementary, 
\[ E_{\psi,\vec{z}}\cap h^{*}\neq\emptyset. \]
Let $r\in E_{\psi,\vec{z}}\cap h^{*}$. Then there is $x\in r$ such that 
\[ r\dststile{\IP_{\max}}{}\psi(\dot{f}(x), \dot{f}(\vec{z})). \]
Hence we have $\dot{f}^{G}(x)$ is in the image of $\pi$ and $\dot{\mathcal{M}}^{G}\models\psi(\dot{f}^{G}(x), \dot{f}^{G}(\vec{z}))$, since $r\in h^{*}\subset G$. This shows $\pi\colon\mathcal{M}_{h^{*}}\to \dot{\mathcal{M}}^{G}$ is elementary. Therefore, $\pi\colon\mathcal{M}_{h^{*}}\to \dot{\mathcal{M}}^{G}$ witnesses that $\Psi(\dot{\mathcal{M}}^{G}, \varphi)$ holds in $V[G][H]$. This is a contradiction! This finishes the proof. 
\end{proof}

%% file: ground_model.tex
\section{Ground model for Main Theorem 1}
In this section, we prove the second main theorem. 
First, we recall notations from \cite{DM1}. 
\begin{defn}
    Let $\lambda$ be a limit of Woodin cardinals and $G\subset\Col(\omega, <\lambda)$ be a $V$-generic filter. Define: 
    \begin{itemize}
        \item $\R^{*}_{G}=\bigcup_{\alpha<\lambda}\R^{V[G\res\alpha]}$, 
        \item $\Hom^{*}_{G}$ is the pointclass of sets $\bigcup_{\alpha<\lambda}A^{G {\res} \alpha}$ for $<\lambda$-universally Baire sets of reals $A$ appearing in $V[G \res \alpha]$ where $\alpha<\lambda$. 
    \end{itemize}
    We call $L(\R^{*}_{G}, \Hom^{*}_{G})$ \textit{a derived model at $\la$}. 
\end{defn}
If $G$ is clear from the context, then we omit the subscript $G$. 
\begin{thm}[The derived model theorem, Woodin, {\cite{DM1}}]
Let $\lambda$ be a limit of Woodin cardinals, let $G\subset\Col(\omega, <\lambda)$ be a $V$-generic filter, and let $L(\R^{*}, \Hom^{*})$ be a derived model at $\la$. Then 
\begin{enumerate}
    \item $L(\R^{*}, \Hom^{*})\models\AD^{+}$, and 
    \item $\Hom^{*}=\{A\subset\R^{*} \mid A \text{ is Suslin, co-Suslin in } L(\R^{*}, \Hom^{*})\}. $
\end{enumerate}
Moreover, if $\la$ is a limit of $<\la$-strong cardinals, then 
\begin{enumerate}
    \item $\mathcal{P}(\R^{*})\cap L(\R^{*}, \Hom^{*})=\Hom^{*}$, and 
    \item $L(\R^{*}, \Hom^{*})\models\AD_{\R}$. 
\end{enumerate}
\end{thm}
The paper \cite{uB} constructs a generalized derived model that satisfies $\ZF+\AD^{+}+$ ``Every set of reals is universally Baire''. 
\begin{defn}[{\cite[Definition 4.4]{uB}}]\label{predF}
    The predicate $F_{\mathrm{uB}}$ consists of all quadruples $(A, Z, p, x)$ such that
    \begin{itemize}
        \item $A$ is universally Baire, $Z$ is a set, $p\in\Col(\omega, Z)$, $x$ is a $\Col(\omega, Z)$-name for a real, and 
        \item $p\dststile{\Col(\omega, Z)}{} x\in A^{\dot{G}}$. 
    \end{itemize}
\end{defn}
\begin{notation}
    In the rest of this paper, the predicate $F_{\mathrm{uB}}$ always denotes the predicate defined in Definition \ref{predF}. 
\end{notation}

The paper \cite[Main Theorem]{uB} shows that the predicate $F_{\mathrm{uB}}$ has enough information to get universally Baire representations. 
\begin{thm}[Larson--Sargsyan--Wilson, {\cite[Main Theorem]{uB}}]\label{ubdm}
    Let $\lambda$ be a limit of Woodin cardinals and a limit of strong cardinals, let $G\subset\Col(\omega, <\lambda)$ be a $V$-generic filter, and define the model 
    \[ \mathcal{M}=(L^{F_{\mathrm{uB}}}(\R^{*}, \Hom^{*}))^{V(\R^{*})}. \]
    Then 
    \[ \mathcal{M}\cap\mathcal{P}(\R^{*})=\Hom^{*}, \]
    and 
    \[ \mathcal{M}\models \AD^{+}+ \text{``Every set of reals is universally Baire''}. \]
\end{thm}
Let us make a few remarks about Theorem \ref{ubdm}. First, note that the model $\mathcal{M}$ also satisfies $\AD_{\R}$. And \cite[Lemma 2.4, Lemma 6.1]{uB} shows that 
\[ V(\R^{*})\models \Hom^{*} \text{ is productive. }\]
We explain the second remark briefly. Given any set $Z$. \cite[Lemma 2.4]{uB} shows that there is an ordinal $\eta$ such that every $\eta$-absolutely complementing pair of trees on $\omega\times\On$ is $Z$-absolutely complementing. Such a pair of trees is obtained by extending the pair of trees coming from the Martin--Solovay construction with strong embeddings. This gives us tree representations compatible with projections. For more details, see \cite[Lemma 4.1, Lemma 4.2]{uB}. 

    The same argument in \cite[Lemma 3.4]{TG} gives us the regularity of $\Theta$. For the definition of being self-iterable, see \cite[Section 2]{TG}. The second author would like to thank Takehiko Gappo for his help in proving the following theorem. 
 \begin{thm}\label{gappo}
        Suppose that $V$ is self-iterable. Let $\lambda$ be an inaccessible cardinal which is a limit of Woodin cardinals and a limit of strong cardinals, and let $G\subset\Col(\omega, <\lambda)$ be a $V$-generic filter. Let 
        \[\mathcal{M}=(L^{F_{\mathrm{uB}}}(\R^{*}, \Hom^{*}))^{V(\R^{*})}. \]
        Then
        \begin{enumerate}
            \item $\mathcal{M}\models\AD^{+}+\AD_{\R}+ ``\Theta \text{ is regular''} + \text{ ``Every set of reals is universally Baire''}$, 
            \item $\mathcal{M}\cap \mathcal{P}(\R^{*}_{G})=\Hom^{*}_{G}$. 
        \end{enumerate}
    \end{thm}

    \begin{proof}
    By Theorem \ref{ubdm}, we have
    \[\mathcal{M}\models\AD^{+}+\AD_{\R}+\text{ ``Every set of reals is universally Baire''}, \]
    and 
    \[\mathcal{M}\cap \mathcal{P}(\R^{*}_{G})=\Hom^{*}_{G}. \]
    We need to check the regularity of $\Theta$. This proof is based on \cite[Section 3]{TG}. The desired result can be obtained by replacing the usual derived model in the arguments in  \cite[Section 3]{TG} with the generalized derived model $L^{F_{\mathrm{uB}}}(\R^{*}_{G}, \Hom^{*}_{G})$. We give an outline for the arguments. 

    Let $\Lg \delta_{i} \mid i<\la\Rg$ be an increasing enumeration of Woodin cardinals below $\la$. 
    Let $x\in\R^{*}_{G}$, and $\delta<\lambda$ be a Woodin cardinal. 
    Say $\delta=\delta_{i+1}$. 
    Let $\mathcal{T}$ be a normal iteration tree on $V$ with last model $W$ based on the window $(\delta_{i}, \delta_{i+1})$ making $x$ generic for the extender algebra at $\delta_{i+1}$. 
    
    The main argument in \cite[Lemma 3.4]{TG} is to recover the derived model of $V$ by $\R^{*}_{G}$-genericity iteration. 
    Let $\Lg a_{j} \mid j<\la\Rg$ be an enumeration of $\R^{*}_{G}$ in $V[G]$. We inductively construct $\Lg W_{j}, \mathcal{U}_{j} \mid j<\la\Rg$ as follows: 
    \begin{enumerate}
        \item $W_{0}=W$, 
        \item for $j<\la$, $W_{j+1}$ is the last model of $\mathcal{U}_{j}$, 
        \item if $j<\la$ is limit, then $W_{j}$ is the well-founded direct limit along the unique cofinal branch through $\bigoplus_{k<j}\mathcal{U}_{k}$, 
        \item for $j<\la$, $\mathcal{U}_{j}$ is an iteration tree on $W_{j}$ according to the tail strategy of $V$ based on the window $(\delta_{i+1+j}, \delta_{i+1+j+1})$ making $a_{j}$ generic over $\mathbb{B}^{W_{j}}_{\pi((\delta_{i+1+j}, \delta_{i+1+j+1}))}$. 
    \end{enumerate}
    Then let $\mathcal{U}=\bigoplus_{j<\la}\mathcal{U}_{j}$, and let $\bar{W}$ be the well-founded direct limit along the unique cofinal branch through $\mathcal{S}=\mathcal{T}^{\smallfrown}\mathcal{U}$. 
    Since $\lambda$ is inaccessible, $\pi^{\mathcal{S}}(\la)=\la$. 
    
    By construction, we pick a $(\bar{W}, \Col(\omega, <\la))$-generic $H\in V[G]$ such that $\R^{*}_{G}=\R^{*}_{H}$. Let $\R^{*}=\R^{*}_{G}=\R^{*}_{H}$. Note that $\bar{W}[H]\subset V[G]$. 
    \cite[Lemma 3.4]{TG} shows 
    \[(\Hom^{*})^{V(\R^{*})}=(\Hom^{*})^{\bar{W}(\R^{*})}. \]
    Hence we have
    \[ L(\R^{*}, \Hom^{*})^{V(\R^{*})}=L(\R^{*}, \Hom^{*})^{\bar{W}(\R^{*})}. \]
    For the same argument in \cite[Lemma 3.4]{TG} (and the rest of proofs in \cite[Section 3]{TG}) to hold, it is enough to show: 

    {\bf Claim.}
    \[(L^{F_{\mathrm{uB}}}(\R^{*}, \Hom^{*}))^{V(\R^{*})}=(L^{F_{\mathrm{uB}}}(\R^{*}, \Hom^{*}))^{\bar{W}(\R^{*})}. \]

    {\bf Proof.} 
    Let $Z\in \bar{W}(\R^{*})$ be a set and let $A\in (\Hom^{*})^{V(\R^{*})}=(\Hom^{*})^{\bar{W}(\R^{*})}$. 
    Let $K\subset\Col(\omega, Z)$ be $V(\R^{*})$-generic. Then we have that 
    \[ (A^{K})^{V(\R^{*})[K]}\cap \bar{W}(\R^{*})[K]=(A^{K})^{\bar{W}(\R^{*})[K]}. \]
    Hence we have that 
    \[ p\dststile{\Col(\omega, Z)}{V(\R^{*})} \dot{x}\in A^{\dot{K}} \iff p\dststile{\Col(\omega, Z)}{\bar{W}(\R^{*})} \dot{x}\in A^{\dot{K}} \]
    where $p\in\Col(\omega, Z)$ and $\dot{x}\in \bar{W}(\R^{*})$ is a $\Col(\omega, Z)$-name for a real. 
    
    Therefore, we have that for all $X\in \bar{W}(\R^{*})$, 
    \[ X\cap (F_{\mathrm{uB}})^{V(\R^{*})}=X\cap (F_{\mathrm{uB}})^{\bar{W}(\R^{*})}. \]
    where $\dot{K}$ is the canonical $\Col(\omega, Z)$-name for a generic. 
    
    It follows that 
    \[(J_{\al}^{F_{\mathrm{ub}}}(\R^{*}, \Hom^{*}))^{V(\R^{*})}=(J_{\al}^{F_{\mathrm{ub}}}(\R^{*}, \Hom^{*}))^{\bar{W}(\R^{*})}\]
    for all $\al$, by induction on the level of $J$-hierarchy. {\bf (Q.E.D. Claim.)}

    For the rest of arguments, see \cite[Section 3]{TG}. 
    This finishes the proof. 
    \end{proof}
    Now we prove the productivity of $\Hom^{*}$ in the smaller model $\mathcal{M}$. Since $F_{\mathrm{uB}}$ only tells how to extend universally Baire sets of reals in generic extensions, it is not trivial that we can find trees compatible with projections in $\mathcal{M}$.  
\begin{lem}\label{prodhom}
Let $\la$, $G$, $\mathcal{M}$ be as in Theorem \ref{gappo}.  
Then 
    \[\mathcal{M}\models ``\Hom^{*} \text{ is productive''. }\] 
\end{lem}
\begin{proof}
   Let $A\in\Hom^{*}$. 
   We may assume that $A\subset\R^{2}$. 
   Let $Z$ be any set in $\mathcal{M}$. We need to check that 
   \[ \dststile{\Col(\omega. Z)}{\mathcal{M}} \exists^{\R}(A^{\dot{H}})=(\exists^{\R}A)^{\dot{H}}. \]
   Let $H\subset\Col(\omega, Z)$ be $V(\R^{*})$-generic. 
   Note that for any $B\in\Hom^{*}$, it follows from the absoluteness of well-foundedness that 
   \[ B^{\mathcal{M}[H]}=B^{V(\R^{*})[H]}\cap \mathcal{M}[H]. \]
   First, we claim that 
   \[\exists^{\R}(A^{\mathcal{M}[H]})\subset (\exists^{\R}A)^{\mathcal{M}[H]}. \]
   This is because 
   \begin{align*}
       \exists^{\R}(A^{\mathcal{M}[H]})&=\exists^{\R}(A^{V(\R^{*})[H]}\cap \mathcal{M}[H])\\
       &=\{x\in \R\cap\mathcal{M}[H] \mid \exists y\in \R\cap \mathcal{M}[H]\, (x, y)\in A^{V(\R^{*})[H]} \} \\
       &\subset \{x\in \R\cap\mathcal{M}[H] \mid \exists y\in \R\, (x, y)\in A^{V(\R^{*})[H]} \} \\
       &=\exists^{\R}(A^{V(\R^{*})[H]})\cap \mathcal{M}[H]\\
       &=(\exists^{\R}A)^{V(\R^{*}[H]}\cap \mathcal{M}[H]\\
       &=(\exists^{\R}A)^{\mathcal{M}[H]}. 
   \end{align*}
We shall show the reverse inclusion 
\begin{align*}
    &\{x\in \R\cap\mathcal{M}[H] \mid \exists y\in \R\, (x, y)\in A^{V(\R^{*})[H]} \}\\
    &\subset\{x\in \R\cap\mathcal{M}[H] \mid \exists y\in \R\cap \mathcal{M}[H]\, (x, y)\in A^{V(\R^{*})[H]} \}. 
\end{align*}  
To prove that, we make use of the uniformization following from $\AD_{\R}$. 
Let $\tilde{A}$ be a totalization of $A$, namely 
\[ \tilde{A}=\{(x, \Lg 0\Rg{}^{\frown} y) \mid (x, y)\in A\}\cup \{(x, \Lg 1,1, 1, \dots\Rg \mid x\notin \exists^{\R} A\}. \]
Let $F\colon \R\to \R$ uniformize $\tilde{A}$. Then by the productivity of $\Hom^{*}$ in $V(\R^{*})$, we have that $F^{V(\R^{*})[H]}$ uniformizes $\tilde{A}^{V(\R^{*})[H]}$. 
Moreover, we have that $F^{\mathcal{M}[H]}=F^{V(\R^{*})[H]}\cap \mathcal{M}[H]$. 
To prove the reverse inclusion, it is enough to show that $F^{\mathcal{M}[H]}$ is a total function with its domain $\R\cap \mathcal{M}[H]$. We shall make use of the Sargsyan's coding trick \cite[Lemma 1.21]{DIMT}. 
Define 
\[ \mathrm{Code}(F):=\{ (x, m, n)\in \R\times \omega^{2} \mid F(x)(m)=n \}. \]
Then by the productivity of $\Hom^{*}$ in $V(\R^{*})$, $\mathrm{Code}(F)^{V(\R^{*})[H]}$ codes the function $F^{V(\R^{*})[H]}$. 

Our next claim is the following. 
\[\dststile{\Col(\omega, Z)}{\mathcal{M}}\forall x\in \R \forall m\in\omega \exists! n\in\omega (x, m, n)\in \mathrm{Code}(F)^{\dot{H}}. \]
{\bf Proof.} First, we prove the existence of such an $n\in\omega$. 
Suppose otherwise. 
Let $(S, T)\in\mathcal{M}$ be a $Z$-absolutely complementing pair of trees projecting to $\mathrm{Code}(F)$ and its complement respectively. 
Let $p\in\Col(\omega, Z)$, let $\rho\in\mathcal{M}^{\Col(\omega, Z)}$, and let $m\in \omega$ be such that 
\[ p\dststile{\Col(\omega, Z)}{\mathcal{M}} \dot{\rho}\in \R \land \forall n\in\omega \, (\dot{\rho}, m, n)\notin \mathrm{Code}(F)^{\dot{H}}. \]
Note that $\mathcal{M}$ satisfies $\DC_{\mathcal{P}(\R)}$, since it satisfies $\AD_{\R}+``\Theta$ is regular''. (See \cite[Theorem 1.3]{Sol}. ) Therefore, it satisfies $\DC$. Let $\eta$ be sufficiently large, and by $\DC$, let $\pi\colon P\to V_{\eta}^{\mathcal{M}}$ be elementary such that  
\begin{itemize}
    \item $P$ is countable transitive, 
    \item $\{ F, S, T, Z, p, \rho\}\subset \ran(\pi)$. 
\end{itemize}
Let $\pi(\Lg \bar{F}, \bar{S}, \bar{T}, \bar{Z}, \bar{p}, \bar{\rho}\Rg)=\Lg F, S, T, Z, p, \rho\Rg$. Let $g\subset\Col(\omega, \bar{Z})$ be $P$-generic with $\bar{p}\in g$. Let $x=\bar{\rho}^g$. 
Then we have for any $n\in\omega$, $(x, m, n)\notin p[\bar{S}]$. Hence for any $n\in\omega$, $(x, m, n)\in p[\bar{T}]$. However, since $p[\bar{T}]\subset p[T]$, we have for any $n\in\omega$, $(x, m, n)\in p[T]$. This is impossible, since $p[S]=\mathrm{Code}(F)$. This is a contradiction! 
The similar argument shows the uniqueness. This finishes the proof of the claim. {\bf (Q.E.D. Claim.)}

Now given any $x\in\R\cap \mathcal{M}[H]$. Define $y\in\R\cap \mathcal{M}[H]$ inside of $\mathcal{M}[H]$
\[ y(m)=n \iff (x, m, n)\in \mathrm{Code}(F)^{\mathcal{M}[H]}. \]
By the claim above, $y$ is well-defined. Since $\mathrm{Code}(F)^{V(\R^{*})[H]}$ really codes $F^{V(\R^{*})[H]}$ and $F^{\mathcal{M}[H]}=F^{V(\R^{*})[H]}\cap \mathcal{M}[H]$, we have $y=F^{\mathcal{M}[H]}(x)$. This proves that $F^{\mathcal{M}[H]}$ is a total function with its domain $\R\cap \mathcal{M}[H]$. 

Therefore, this gives the reverse inclusion
\begin{align*}
    &\{x\in \R\cap\mathcal{M}[H] \mid \exists y\in \R\, (x, y)\in A^{V(\R^{*})[H]} \}\\
    &\subset\{x\in \R\cap\mathcal{M}[H] \mid \exists y\in \R\cap \mathcal{M}[H]\, (x, y)\in A^{V(\R^{*})[H]} \}. 
\end{align*} 
Hence $\exists^{\R}(A^{\mathcal{M}[H]})=(\exists^{\R}A)^{\mathcal{M}[H]}$. This finishes the proof. 
\end{proof}
We showed that 
\begin{thm}
        Suppose that $V$ is self-iterable. Let $\lambda$ be an inaccessible cardinal which is a limit of Woodin cardinals and a limit of strong cardinals, and let $G\subset\Col(\omega, <\lambda)$ be a $V$-generic filter. 
        Let 
        \[\mathcal{M}=(L^{F_{\mathrm{uB}}}(\R^{*}, \Hom^{*}))^{V(\R^{*})}. \]
        Then
        \begin{enumerate}
            \item $\mathcal{M}\models\AD^{+}+\AD_{\R}+ ``\Theta \text{ is regular''} + \text{ ``Every set of reals is universally Baire''}$, 
            \item $\mathcal{M}\cap \mathcal{P}(\R^{*}_{G})=\Hom^{*}_{G}$, 
            \item $\mathcal{M}\models``\Gamma^{\infty}(=\Hom^{*}_{G})\text{ is productive''}$.  
        \end{enumerate}
    \end{thm}
    \nocite{*}